\documentclass[a4paper, 10pt, reqno]{amsart}

\usepackage{amssymb, amsmath, amsthm}
\usepackage{graphicx, psfrag, setspace, subfigure, gensymb}
\usepackage{amsfonts}
\usepackage{bbm}
\usepackage{fix-cm}
\usepackage{comment}
\usepackage{tikz,tikz-cd}
\usetikzlibrary{matrix,arrows,decorations.pathmorphing,decorations.pathreplacing}
\tikzset{commutative diagrams/diagrams={baseline=-2.5pt},commutative diagrams/arrow style=tikz}
\usepackage[colorlinks]{hyperref}
\usepackage{float}
\usepackage{setspace} 

\newcommand\A{\mathbb{A}}

\newcommand{\cA}{\mathcal{A}}

\DeclareMathOperator{\cok}{cok}
\newcommand{\set}[1]{\left\{{#1}\right\}}

\newcommand\id{\mathrm 1}

\newcommand\End{\operatorname{End}}

\newcommand\Ext{\operatorname{Ext}}

\renewcommand\P{\mathbb P}

\newcommand{\Spec}{\operatorname{Spec}}

\newcommand{\beq}[1]{\begin{equation}\label{#1} }
\newcommand{\eeq}{\end{equation}}
\newcommand{\pgap}{\vspace{5pt}}

\newtheorem{prop}[equation]{Proposition}
\newtheorem{thm}[equation]{Theorem}
\newtheorem{lem}[equation]{Lemma}

\newtheorem{defn}[equation]{Definition}

\newtheorem{cor}[equation]{Corollary}

\makeatletter \@addtoreset{equation}{section} \makeatother

\let\oldtocsection=\tocsection
\let\oldtocsubsection=\tocsubsection
\let\oldtocsubsubsection=\tocsubsubsection
\renewcommand{\tocsection}[3]{\hspace{0em}\oldtocsection{#1}{#2}{#3}}
\renewcommand{\tocsubsection}[3]{ \hspace{1em} \oldtocsubsection{#1}{\small{#2}}{\small{#3}} }
\renewcommand{\tocsubsubsection}[3]{\hspace{2em}\oldtocsubsubsection{#1}{\small{#2}}{\small{#3}}}

\newcommand{\marginparstretch}{0.6}
\let\oldmarginpar\marginpar
\renewcommand\marginpar[1]{\-\oldmarginpar[\framebox{\setstretch{\marginparstretch}\begin{minipage}{\marginparwidth}{\raggedleft\scriptsize #1}\end{minipage}}]{\framebox{\setstretch{\marginparstretch}\begin{minipage}{\marginparwidth}{\raggedright\scriptsize #1}\end{minipage}}}}

\begin{document}

\title{A non-commutative Bertini theorem}

\author{J{\o}rgen Vold Rennemo}
\address{University of Oslo \\ Blindern \\ N-0315 Oslo \\ Norway}
\email{jvrennemo@gmail.com}

\author{Ed Segal}
\address{University College London \\ 25 Gordon St \\ London WC1H  0AY\\ UK}
\email{e.segal@ucl.ac.uk}

\author{Michel Van den Bergh}
\address{Universiteit Hasselt \\ Universitaire Campus \\ B-3590 Diepenbeek \\ Belgium}
\email{michel.vandenbergh@uhasselt.be}
\maketitle

\begin{abstract}
We prove a version of the classical `generic smoothness' theorem with smooth varieties replaced by non-commutative resolutions of singular varieties.
This in particular implies a non-commutative version of the Bertini theorem.
\end{abstract}

\section{Introduction}
Let $k$ be a field of characteristic 0 and let $X$ be a smooth variety over $k$, equipped with a morphism $f: X \to \P^n$ for some $n$. 
The classical Bertini Theorem, in one of its formulations, states that for a general hyperplane $H\subset \P^n$ the fibre-product $X_H = X\times_{\P^n} H$ will still be smooth. The aim of this short note is to establish a non-commutative analogue of this theorem, where instead of assuming $X$ is smooth, we assume it comes provided with a \emph{non-commutative resolution}.

For an affine variety  $X = \Spec S$, we use the following definition of non-commutative resolution.

\begin{defn}[\cite{vandenbergh_noncommutative_2004, spenko_non-commutative_2015}]
Let $S$ be a normal Noetherian domain. A \emph{non-commutative resolution} of $S$ is an $S$-algebra of finite global dimension of the form $A = \End_{S}(M)$ where $M$ is a non-zero finitely generated reflexive $S$-module. 
The resolution is said to be \emph{crepant} if in addition $S$ is Gorenstein and $A$ is a maximal Cohen--Macaulay $S$-module.
\end{defn}

The idea is that the category $A$-mod should be a reasonable substitute for the category of coherent sheaves on a geometric resolution of $\Spec S$. The most important property is the requirement that $A$ has finite global dimension, this is `smoothness' of the resolution. The fact that $A$ is `birational' to $S$ is also encoded in this definition: $M$ must be locally-free (and non-zero) over some Zariski open subset $U\subset \Spec S$, then $A_U = \End(M_U)$ is a trivial Azumaya algebra and is Morita equivalent to $S_U$. In fact if the resolution is crepant one can prove that $M$ must be locally-free over the smooth locus in $\Spec(S)$. This is no longer true for non-crepant resolutions, see \cite{dao_noncommutative_2016}.

We abbreviate \emph{non-commutative resolution} and \emph{non-commutative crepant resolution} to NCR and NCCR respectively. When $X$ is not affine, we simply replace the single algebra $A$ with a sheaf of algebras.
\begin{defn} Let $X$ be a normal variety. An NCR (resp.~NCCR) of $X$ is a coherent sheaf of algebras $A$ on $X$, such that for any affine open subset $U = \Spec S \subset X$, the algebra $A|_{U}$ is an NCR (resp.~NCCR) of $S$.
\end{defn}

Our path to the Bertini theorem follows a standard approach, similar to the commutative case, and passes through some other results which are interesting in their own right. The starting point is a non-commutative version of ``generic smoothness'': the geometrically intuitive result that if $f \colon X \to Y$ is a map of varieties and $X$ is nonsingular, then for a generic point $y \in Y$ the fibre  $f^{-1}(y)$ is nonsingular.

\begin{thm}\label{ncGenericSmoothness}
Let $k$ be a field of characteristic zero,  let $R$ be a finitely-generated Noetherian domain over $k$, and let $S$ be a commutative finitely-generated  normal $R$-algebra.  Let $A$ be a non-commutative resolution of $S$. Then for a general point $p \in \Spec(R)$, the algebra $A_{k(p)}$ is an NCR of $S_{k(p)}$, which is moreover crepant if $A$ is.
\end{thm}

Here $k(p)$ denotes the residue field at $p$. When we say `for a general point' here we mean that there is a Zariski open subset $U\subset \Spec(R)$ such that $A_{k(p)}$ is an NCR for every $p\in U$. 

This theorem is the technical contribution of this paper, our other results are easy corollaries. Its proof is not difficult, but neither is it purely formal, since it does require that the ground field $k$ has characteristic 0. Indeed, without further assumptions, none of the results are true in characteristic $p$ even in the commutative setting. There is an elementary and well-known counterexample: set $A=S=k[x]$ and $R=k[t]$, and consider the homomorphism $t \mapsto x^p$. 
If char$(k)=p$ then for every $y \in \Spec R$ the fibre $A_y$ has infinite global dimension.
\pgap

From Theorem \ref{ncGenericSmoothness}, we easily deduce an `affine non-commutative Bertini's theorem'. Stating an affine version of Bertini's theorem takes a little care, even for complex varieties. There are examples of resolutions $\pi: X\to \Spec S$ where $S$ is local, such that for any $f\in S$ lying in the maximal ideal the slice $\set{\pi^*f=0}$ is singular.
\begin{thm} \label{ncBertiniAffine}
Let $S$ be a normal Noetherian domain over a field $k$ of characteristic zero, and let $A$ be an NCR of $S$. Fix a finite-dimensional $k$-vector space $V\subset S$ with $1\in V$. Then for a general element $f\in V$, the algebra $A/fA$ is an NCR of $S/fS$, which is crepant if $A$ is.
\end{thm}

Finally, we turn to our main motivation, which is the non-commutative version of the projective Bertini theorem.

Let $(X, A)$  be a normal $k$-variety equipped with a non-commutative resolution, and let $f \colon X \to \P^n$ be a morphism. Given a hyperplane $H\subset \P^n$, the fibre-product $X_H$ carries a sheaf of algebras $A_H$ by restriction.
\begin{thm}\label{ncBertiniProj} 
For a general hyperplane $H \subset \P^{n}$, the pair $(X_H, A_H)$ is a non-commutative resolution of $X_H$, which is crepant if $(X,A)$ is.
\end{thm}

The first- and second-named authors came across this question in the context of a specific NCCR of a Pfaffian variety \cite{rennemo_hori-mological_2016}; see that paper for an application of these results.

No doubt these results still hold if we replace $A$ with an appropriate sheaf of smooth categories (or dg-categories) $\cA$ on $X$, rather than assuming that $\cA=A$-mod as we do here. 
This would allow for example modules over a sheaf of DGAs, or modules twisted by a gerbe. However, we didn't feel the need to work in that level of generality for this paper.

\subsection{Acknowledgements} We thank Michael Wemyss for helpful conversations. JVR would like to acknowledge the contribution of Icelandair, who stranded him in Reykjavik for a very productive 24 hours.
\pgap

This project has received funding from the European Research Council (ERC) under the European Union Horizon 2020 research and innovation programme (grant agreement No.725010).

\section{Proof of theorems}
\label{sec:proofs}

The bulk of this (rather skinny) paper is the proof of the following proposition:

\begin{prop}\label{genericFGD} 
Let $R$ be a Noetherian domain whose fraction field $K(R)$ is of characteristic 0, let $S$ be a finitely generated commutative $R$-algebra, and suppose $A$ is an $S$-algebra which is finite rank as an $S$-module.

If $A$ has finite global dimension, then for a generic $p\in \Spec R$, the algebra $A_{k(p)}$ has finite global dimension.
\end{prop}

One can compare this to Schofield's result that for algebras which are finite-dimensional over a field,  having finite global dimension is a Zariski open condition \cite{schofield_bounding_1985}.

Our proof works by flipping between `finite global dimension' and the following property of `non-commutative smoothness', whose definition appears to be due to Kontsevich. Recall that a module is perfect if it has finite projective dimension.

\begin{defn} Let $R$ be a commutative ring, and let $A$ be an $R$-algebra. We declare that $A$ is \emph{smooth over $R$} if $A$ is flat over $R$ and $A$ is perfect as a module over $A^{op}\otimes_R A$. 
\end{defn}

In other words, if we work relative to $R$ then the diagonal bimodule is perfect.  The flatness assumption is just to avoid needing to work with the derived tensor product $A^{op}\otimes^L_R A$.

\begin{proof}[Proof of Proposition \ref{genericFGD}] Here we just give the outline of the argument, the necessary supporting lemmas are in Section \ref{sec.lemmas}. Let $K = K(R)$ and let $\overline{K}$ be its algebraic closure. We prove our claim via a chain of implications $(i) \Rightarrow (i+1)$. Observe that all the algebras involved, $A$, $A^{op}$, $A_K$,\ldots are finite rank over a commutative Noetherian ring so they are Noetherian (Lemma \ref{lem.Noetherian}).
\begin{enumerate}
\item $A$ has finite global dimension.
\item $A_{K}$ has finite global dimension.

This is Lemma \ref{thm:globalDimensionPreservedByFlatBaseChange}.

\item $A_{\overline{K}}$ has finite global dimension.

 This follows from Lemma \ref{separableFGD} -- the extension $\overline{K}/K$ is separable because $\operatorname{char}(K)=0$.

\item $A_{\overline{K}}$ is smooth over $\overline{K}$.

By  Lemma \ref{lem.AopFGD} the algebra $A^{op}_{\overline{K}}$ has finite global dimension. Then   $A_{\overline{K}}^{op} \otimes_{\overline{K}} A_{\overline{K}}$ also has finite global dimension by Lemma \ref{productFGD}, in particular the diagonal bimodule is perfect.

\item $A_{K}$ is smooth over $K$.

 This is  Lemma \ref{thm:smoothnessDescends}.

\item There is an $f\in R$ such that $A_f$ is smooth over $R_f$. 

By Corollary \ref{thm:FiniteProjectiveDimensionIsZariskiOpen} (with $T=S\otimes_R S$, $\Lambda = A^{op} \otimes_{R} A$ and $M = A$) there is an $f\in R$ such that $A_f$ is perfect over $A^{op}_f\otimes_{R_f} A_f$. 
By generic freeness (Lemma \ref{thm:GrothendieckGenericFreeness}) we may choose $f$ so that $A_{f}$ is flat over $R_{f}$, and so $A_f$ is smooth over $R_f$. 

\item $A_{k(p)}$ is smooth over $k(p)$ for a general $p \in \Spec S$.

 Take $f$ as above and let $U=\Spec(R_f)$. Take a bounded resolution of $A_f$ by finitely generated projective $A_f^{op}\otimes_{R_f} A_f$ modules. By shrinking $U$ further we may assume that all these modules are flat over $U$, then restricting to $p\in U$ we have a bounded projective resolution of $A_{k(p)}$. 

\item $A_{k(p)}$ has finite global dimension for a general $p \in \Spec S$. 

We can get a projective resolution of any  $A_{k(p)}$-module by tensoring it with our projective resolution of the diagonal $A_{k(p)}$-bimodule.
\end{enumerate}

\end{proof}

Assume now that we are in the setting of Theorem \ref{ncGenericSmoothness}:
\begin{itemize}
\item $k$ is a field of characteristic zero,
\item $R$ is a finitely-generated Noetherian domain over $k$,
\item $S$ is a commutative finitely-generated normal algebra over $R$, and
\item $M$ is a reflexive $S$-module such that $A = \End_{S}(M)$ is an NCR of $S$.
\end{itemize}

We want to prove that for a general point $p\in \Spec R$, the algebra $A_{k(p)}$ is an NCR of $S_{k(p)}$, which is crepant if $A$ is. We've already discussed the most important aspect -- the finiteness of the global dimension -- but there are a few remaining details.

\begin{lem}
\label{thm:genericPropertiesEasy}
For a general point $p \in \Spec R$, $S_{k(p)}$ is normal, $M_{k(p)}$ is a reflexive $S_{k(p)}$-module, and $A_{k(p)} = \End_{S_{k(p)}}(M_{k(p)})$.

If $S$ is Gorenstein, then so is $S_{k(p)}$, and if $A$ is a maximal Cohen--Macaulay $S$-module, then $A_{k(p)}$ is a maximal Cohen--Macaulay $S_{k(p)}$-module.
\end{lem}
\begin{proof}
Take a presentation for $M$, \emph{i.e.} let $M = \cok \phi$ for $\phi \colon S^{m} \to S^{n}$.
The fact that $A = \End_{S}(M)$ means that $A = \ker (\phi^{T} \otimes id_{M})$, so there is an exact sequence
\[
0 \to A \to M^{n} \to M^{m} \to N := \cok (\phi^{T} \otimes id_{M}).
\]
By generic freeness (Lemma \ref{thm:GrothendieckGenericFreeness}), we may assume (after passing to an open set in $\Spec R$) that $A$, $M$ and $N$ are flat $R$-modules.
It follows that $A_{k(p)}$ is the kernel of $M_{k(p)}^{n} \to M_{k(p)}^{m}$, which means precisely that $A_{k(p)} = \End_{S_{k(p)}}(M_{k(p)})$.

The other properties are all consequences of \cite[Thm.~3.3.10, Thm.~3.3.15]{flenner_joins_1999}.
\end{proof}

\begin{proof}[Proof of Theorem \ref{ncGenericSmoothness}]
Combine Lemma \ref{thm:genericPropertiesEasy} and Prop.~\ref{genericFGD}.
\end{proof}

From here our `non-commutative Bertini' theorems follow easily.

\begin{proof}[Proof of Theorem \ref{ncBertiniAffine}]
Let $S$ be an normal, finitely-generated $k$-algebra and let $A$ be an NCR of $S$. Fix $V = \langle 1, f_1,..., f_n\rangle\subset S$. Now set  $R=k[x_{0},..., x_{n}]$, and let
$$S' = \frac{S\otimes_k k[x_0, ..., x_n]}{(x_{0}+ x_{1}f_{1} + \cdots + x_{n}f_{n})} \cong S\otimes_k k[x_1,..., x_n]. $$
Then $S'$ is an integral normal algebra over $R$, and $A' = A\otimes_S S' $ is an NCR of $S$. Given a point $f\in V = \Spec R$, the fibres of $S$ and $A$ over $f$ are $S/fS$ and $A/fA$. Hence by Thm.~\ref{ncGenericSmoothness}, for a general element $f\in V$ the algebra $A/fA$ is an NCR of $S/fS$, and it's an NCCR if $A$ is an NCCR.
\end{proof}

\begin{proof}[Proof of Theorem \ref{ncBertiniProj}]
A hyperplane in $\P^{n}$ is a divisor defined by an equation $\sum_{k=0}^{n}a_{i}x_{i}$, with $a_{i} \in k$, so the result follows from Theorem \ref{ncBertiniAffine} by seeing $\P^{n}$ as a union of $n+1$ copies of $\A^{n}$.  
\end{proof}

\section{Technical lemmas} \label{sec.lemmas}
In this section we provide the lemmas necessary for the proof of Proposition \ref{genericFGD}.
The argument is spelled out in detail for the benefit of those readers who are not very familiar with non-commutative algebras (the first two authors of this paper themselves belonging to this group), and in the hope of making it clear where each hypothesis is used.

\begin{lem}\label{lem.Noetherian} Let $S$ be a Noetherian commutative ring, and let $A$ be an $S$-algebra which has finite rank as an $S$-module. Then $A$ is Noetherian.
\end{lem}
\begin{proof}
Any ideal of $A$ must be a finitely-generated $S$-module, hence a finitely-generated $A$-module.
\end{proof} 

\begin{lem}
\label{thm:globalDimensionPreservedByFlatBaseChange}
Let $A$ be an algebra over an integral domain $R$ such that $A$ has finite global dimension. If $A_{K(R)}$ is Noetherian then it has finite global dimension.
\end{lem}
\begin{proof}
Let $K = K(R)$. Every finitely generated $A_{K}$-module $M$ is the cokernel of some $\phi \colon A_{K}^{m} \to A_{K}^{n}$.
Multiplying by denominators, we can ensure that $\phi$ is the localisation of a map $\phi' \colon A^{m} \to A^{n}$, and so setting $M' = \cok(\phi')$ we find $M_{K}' = M$.
Since $M'$ is perfect, so is $M$.  
\end{proof}

We say that an infinite field extension $L/K$ is separable if it is algebraic and every finite subextension is separable.

\begin{lem}\label{separableFGD} Let $A$ be an algebra over a field $K$. Let $L$ be a separable field extension of $K$, and let $A_L = A\otimes_K L$. If $A$ has finite global dimension, and $A_{L}$ is Noetherian, then $A_L$ has finite global dimension.
\end{lem}
This is mostly contained in \cite[Thm.~2.4]{jensen_homological_1982}, but we give a complete proof here for convenience. 
\begin{proof}
Say $A$ has global dimension $n$, and let $M$ be a finite-rank $A_L$-module. Since $A_{L}$ is Noetherian, $M$ is finitely presented, and hence defined over some finite extension $L'\subset L$ of $K$, i.e.~$M=M'\otimes_{L'}L$ for an $A_{L'}$-module $M'$. 
Separability implies that the map $L'\otimes_K L' \to L'$ splits as a map of $L'$-bimodules (see e.g.~\cite[Thm.~5.3.7, Prop.~5.3.16]{rowen_ring_88}), so $M'$ is a direct $A_{L'}$-module summand of
\begin{align*}
L' \otimes_{K} L' \otimes_{L'} M' = L' \otimes_{K} M' = A_{L'} \otimes_{A} M'.
\end{align*}
We thus get the inequality of projective dimensions
\[
pd_{A_{L'}} (M') \le pd_{A_{L'}} (A_{L'} \otimes_{A} M') \le pd_{A} (M') \le n. \qedhere
\]
\end{proof}

\begin{lem}\label{lem.AopFGD}\cite[Cor.~5]{auslander_dimension_1955} If both $A$ and $A^{op}$ are Noetherian, and $A$ has finite global dimension, then $A^{op}$ has finite global dimension.
\end{lem}

Recall that a \emph{polynomial identity ring} is a ring $A$ satisfying a polynomial identity, which means that there is a non-commutative polynomial $f \in k\langle x_{1}, \ldots, x_{n} \rangle$ such that for any $n$ elements $a_{1},\ldots,a_{n} \in A$ we have $f(a_{1},\ldots,a_{n}) = 0$.
In particular this class includes commutative rings (by taking $f = x_{1}x_{2}-x_{2}x_{1}$), and more generally any ring which is a finite module over its centre.

\begin{lem}\cite[Lem. 4.2.]{stafford_noncommutative_2008}
\label{productFGD}
Let $K$ be an algebraically closed field, and let $A_1, A_2$ be two finitely-generated polynomial identity algebras over $K$. 
If $A_1$ and $A_2$ both have finite global dimension, then so does $A_1\otimes_K A_2$.
\end{lem}

\begin{lem}
\label{thm:smoothnessDescends}
An algebra $A$ over a field $K$ is smooth over $K$ if and only if $A_{\overline{K}}$ is smooth over $\overline K$.
\end{lem}
\begin{proof}
If $A$ is a perfect $A^{op} \otimes_{K} A$-module, then obviously $A_{\overline{K}}$ is perfect. For the reverse implication, suppose for a contradiction that $A$ is not smooth.
Then there exists a sequence of $A^{op} \otimes_{K} A$-modules $M_{i}$ and an increasing sequence of integers $n_{i}$ such that $\Ext^{n_{i}}_{A^{op} \otimes_{K} A}(A, M_{i}) \not=0$.
But then $\Ext^{n_{i}}_{A^{op} \otimes_{\overline{K}} A}(A_{\overline{K}}, (M_{i})_{\overline{K}}) \not=0$, contradicting the assumption that $A_{\overline{K}}$ is smooth over $\overline{K}$.
\end{proof}

Recall Grothendieck's generic freeness lemma \cite[Thm.~14.4]{eisenbud_commutative_1995}:
\begin{lem} 
\label{thm:GrothendieckGenericFreeness}
Let $R$ be a Noetherian domain, and $T$ a finitely generated commutative $R$-algebra.
Let $M$ be a finitely generated $T$-module. 
Then there is an element $f \in R$ such that $M_{f}$ is a free $R_{f}$-module. In particular if $M_{K(R)} = 0$ then there exists an $f$ such that $M_{f} = 0$.
\end{lem}

\begin{lem}
\label{thm:OurGenericFreeness}
Let $R$ be a Noetherian domain, let $T$ be a finitely generated commutative $R$-algebra, and let $\Lambda$ be a $T$-algebra which is finitely generated as a $T$-module. Let $P$ be a finitely generated $\Lambda$-module such that $P_{K(R)}$ is a projective $\Lambda_{K(R)}$-module. Then there exists an $f \in R$ such that $P_{f}$ is a projective $\Lambda_{f}$-module.
\end{lem}
\begin{proof}
Let $K = K(R)$.
Since $P_{K}$ is projective, there exists a free module $F = \Lambda^{\oplus n}$ and maps $i \colon P_{K} \to F_{K}$ and $p \colon F_{K} \to P_{K}$, such that $pi = \id_{P_{K}}$.
There exists an $f \in R$ such that the maps are defined over $R_{f}$, i.e.~there are maps $i_{f} \colon P_{f} \to F_{f}$ and $p_{f} \colon F_{f} \to P_{f}$ which localise to $i$ and $p$.

Now $p_{f}i_{f} - \id_{P_{f}}$ has a kernel $L$ and a cokernel $C$, which are finitely generated $T$-modules.
We have $L_{K} = C_{K} = 0$, and so by Lemma \ref{thm:GrothendieckGenericFreeness}, there exists a $g \in R_{f}$ such that $L_{g} = C_{g} = 0$.
This implies that $p_{fg}i_{fg} = \id_{P_{fg}}$, and so $P_{fg}$ is a projective $\Lambda_{fg}$-module.
\end{proof}

\begin{cor}
\label{thm:FiniteProjectiveDimensionIsZariskiOpen}
Let $R,T$ and $\Lambda$ be as in the previous lemma.
If $M$ is a finitely generated $\Lambda$-module such that $M_{K(R)}$ is a perfect $\Lambda_{K(R)}$-module, then there exists an $f \in R$ such that $M_{f}$ is a perfect $\Lambda_{f}$-module.
\end{cor}
\begin{proof}
Take a projective resolution $\cdots \to M^{i} \to M^{i-1} \to \cdots$ of $M$, and let $C^{n}$ be the kernel of $M^{n} \to M^{n-1}$.
Since $\Lambda$ is Noetherian (Lemma \ref{lem.Noetherian}), we may assume that all $M^{i}$ and $C^{i}$ are finitely presented.
Since $M_{K}$ has finite projective dimension, there is an $n$ such that $C^{n}_{K}$ is a projective $\Lambda_{K}$-module. 
Applying Lemma \ref{thm:OurGenericFreeness}, we find an $f$ such that $C^{n}_{f}$ is projective, and so $C^{n}_{f} \to M^{n-1}_{f} \to \cdots \to M^{0}_{f}$ is a projective resolution of $M_{f}$.
\end{proof}

\bibliographystyle{alpha}

\bibliography{bibliography}

\end{document}